\UseRawInputEncoding
\documentclass[12pt]{article}

\usepackage{amsfonts,amsmath,amsthm}
\usepackage{amssymb}
\usepackage{algorithm}
\usepackage{algpseudocode}
\usepackage{xcolor}
\usepackage{graphicx}
\usepackage{stmaryrd}        
\usepackage{multirow}
\usepackage{dsfont}

\usepackage[paper=a4paper,dvips,top=2cm,left=2cm,right=2cm,
foot=1cm,bottom=4cm]{geometry}

\newcommand{\ve}{{\bf e}}

\newcommand{\vw}{{\bf w}}
\newcommand{\vx}{{\bf x}}
\newcommand{\vy}{{\bf y}}

\newcommand{\vz}{{\bf z}}

\newcommand{\R}{\mathbb{R}}

\newcommand{\A}{\mathcal{A}}
\newcommand{\I}{\mathcal{I}}
\newcommand{\B}{\mathcal{B}}

\newtheorem{Thm}{Theorem}[section]
\newtheorem{Def}[Thm]{Definition}

\newtheorem{Prop}[Thm]{Proposition}
\newtheorem{Cor}[Thm]{Corollary}
\newtheorem{Conj}[Thm]{Conjecture}

\begin{document}
	
	\title{Sum of Squares Decompositions for Structured Biquadratic Forms}
    \author{Yi Xu\footnote{School of Mathematics, Southeast University, Nanjing  211189, China. Nanjing Center for Applied Mathematics, Nanjing 211135,  China. Jiangsu Provincial Scientific Research Center of Applied Mathematics, Nanjing 211189, China. ({\tt yi.xu1983@hotmail.com})}
		\and
		Chunfeng Cui\footnote{School of Mathematical Sciences, Beihang University, Beijing  100191, China.
			({\tt chunfengcui@buaa.edu.cn})}
		\and {and \
			Liqun Qi\footnote{
				Department of Applied Mathematics, The Hong Kong Polytechnic University, Hung Hom, Kowloon, Hong Kong.
				({\tt maqilq@polyu.edu.hk})}
		}
	}
	
	\date{\today}
	\maketitle
	
	\begin{abstract}
		This paper studies sum-of-squares (SOS) representations for structured biquadratic forms. We prove that diagonally dominated symmetric biquadratic tensors are always SOS. For the special case of symmetric biquadratic forms, we establish necessary and sufficient conditions for positive semi-definiteness of monic symmetric biquadratic forms, characterize the geometry of the corresponding PSD cone as a convex polyhedron, and prove that every such PSD form is SOS for any dimensions $m$ and $n$. We also formulate conjectures regarding SOS representations for symmetric M-biquadratic tensors and symmetric $\mathrm{B}_{0}$-biquadratic tensors, discussing their likelihood and potential proof strategies. Our results advance the understanding of when positive semi-definiteness implies sum-of-squares decompositions for structured biquadratic forms.
		
		\medskip
		
		\textbf{Keywords.} Biquadratic forms, sum-of-squares, positive semi-definiteness, M-eigenvalues, symmetric bilinear forms.
		
		\medskip
		\textbf{AMS subject classifications.} 11E25, 12D15, 14P10, 15A69, 90C23.
	\end{abstract}
	
	\section{Introduction}
	
	A fundamental question at the intersection of algebra and optimization is whether a multivariate polynomial that is nonnegative everywhere (positive semi-definite, or PSD) can be written as a sum of squares (SOS) of polynomials. Hilbert \cite{Hi88} famously proved that this is not always true, except in a few special cases. This dichotomy between nonnegativity and a structured SOS representation has profound implications, particularly because the existence of an SOS decomposition makes the problem of verifying nonnegativity computationally tractable via semidefinite programming.
	
	While the general problem remains open, significant research has focused on identifying broad classes of PSD polynomials and tensors that are guaranteed to be SOS. For instance, many structured tensors are known to be SOS, including even-order weakly diagonally dominant tensors, symmetric M-tensors, and several other families mentioned in \cite{QL17}. This paper focuses on a specific and important class of polynomials: \textbf{biquadratic forms}.
	
	An $m \times n$ biquadratic form is a polynomial in two sets of variables, $\mathbf{x} \in \mathbb{R}^m$ and $\mathbf{y} \in \mathbb{R}^n$, that is quadratic in $\mathbf{x}$ when $\mathbf{y}$ is fixed, and quadratic in $\mathbf{y}$ when $\mathbf{x}$ is fixed. Such a form can be represented by a fourth-order tensor $\mathcal{A} = (a_{ijkl})$ as:
	\[
	P(\mathbf{x}, \mathbf{y}) = \sum_{i,k=1}^{m} \sum_{j,l=1}^{n} a_{ijkl} x_i x_k y_j y_l.
	\]
	A PSD biquadratic form is one for which $P(\mathbf{x}, \mathbf{y}) \geq 0$ for all $\mathbf{x}, \mathbf{y}$. It is SOS if it can be written as a finite sum of squares of bilinear forms.
	
	The question of whether every PSD biquadratic form is SOS was settled negatively by Choi \cite{Ch75}, who provided an explicit $3 \times 3$ counterexample. However, recent work \cite{QCCX25} has shown that several important subclasses, such as weakly completely positive biquadratic tensors, are indeed SOS. This positive result raises a natural and important question: \textbf{What other structured PSD biquadratic tensors are SOS?}
	
	In particular, the authors of \cite{QC25} recently identified and analyzed several new classes of PSD biquadratic tensors, including \textbf{diagonally dominated symmetric}, \textbf{symmetric M-}, and \textbf{symmetric $\mathrm{B}_{0}$-biquadratic tensors}. A central question, previously unanswered, is whether the PSD tensors in these classes also admit SOS representations.
	
	This paper provides a comprehensive study of this question. Our main contributions are:
	\begin{enumerate}
		\item We prove that \textbf{every diagonally dominated symmetric biquadratic tensor is SOS} (Theorem \ref{thm:main1}).
		\item We formulate conjectures on the SOS property for symmetric M-biquadratic tensors (Conjecture \ref{conj:M-sos}) and symmetric $\mathrm{B}_{0}$-biquadratic tensors (Conjecture \ref{conj:B0-sos}), discussing the likelihood of their resolution.
		\item We focus on the special case of \textbf{symmetric biquadratic forms}. Notably, Choi's counterexample \cite{Ch75} is \textit{not} symmetric. We establish necessary and sufficient conditions for positive semi-definiteness of monic symmetric biquadratic forms (Theorem \ref{thm:monic-psd}), characterize the geometry of the corresponding PSD cone as a convex polyhedron (Theorem \ref{thm:psd-polyhedron}), and \textbf{prove that every symmetric PSD biquadratic form is SOS for any dimensions $m$ and $n$} (Theorem \ref{thm:monic-sym-sos}). We conjecture that $3 \times 3$ symmetric forms can be expressed as the sum of at most five squares of bilinear forms (Conjecture \ref{conj:sos-rank-3x3-sym}).
	\end{enumerate}
	
	The rest of this paper is organized as follows. Section 2 is devoted to diagonally dominated tensors, where we prove Theorem \ref{thm:main1}. Section 3 presents conjectures on the SOS property for symmetric M-biquadratic tensors and symmetric $\mathrm{B}_{0}$-biquadratic tensors. Section 4 formally defines $x$-symmetric and $y$-symmetric biquadratic forms and studies monic symmetric biquadratic forms, including the general SOS result. Section 5 offers further remarks and discusses directions for future research.
	
	We denote $\{1, \ldots, m\}$ as $[m]$.
	
	\section{Diagonally Dominated Symmetric Biquadratic Tensors}
	
	Suppose that $\A = (a_{ijkl})$, where $a_{ijkl} \in \R$ for $i, k \in [m], j, l \in [n]$. Then $\A$ is called an $m \times n$ {\bf biquadratic tensor}. If
	\begin{equation} \label{e1}
		a_{ijkl} = a_{kjil} = a_{klij}
	\end{equation}
	for $i, k \in [m], j, l \in [n]$, then $\A$ is called a {\bf symmetric biquadratic tensor}.
	Denote the set of all $m \times n$ biquadratic tensors by $BQ(m, n)$, and the set of all $m \times n$ symmetric biquadratic tensors by $SBQ(m, n)$.
	
	If the entries of $\A$ satisfy
	\begin{equation}\label{dom1}
		a_{ijij}\ge r_{ij}\equiv  {1 \over 2}\sum_{i_2=1}^m \left(\sum_{j_2=1}^n \left|\bar a_{iji_2j_2}\right| + \sum_{j_1=1}^n \left|\bar a_{ij_1i_2j}\right|\right),
	\end{equation}
	for all $i \in [m]$ and $j \in [n]$,
	then $\A$ is called a {\bf diagonally dominated biquadratic tensor}.
	It was shown in \cite{QC25} that a diagonally dominated symmetric biquadratic tensor is a PSD tensor.
	
	\begin{Thm}\label{thm:main1}
		A diagonally dominated symmetric biquadratic tensor is an SOS tensor.
	\end{Thm}
	
	\begin{proof}
		Let $\A = (a_{ijkl})$ be a diagonally dominated symmetric biquadratic tensor. Consider the $mn \times mn$ matrix $M$ defined by
		\[
		M_{(i,j),(k,l)} = a_{ijkl}, \quad \text{for } i,k \in [m],\; j,l \in [n],
		\]
		where rows and columns are indexed by pairs $(i,j) \in [m] \times [n]$.
		
		The symmetry conditions $a_{ijkl} = a_{kjil} = a_{klij}$ ensure that $M$ is symmetric. Indeed:
		\begin{itemize}
			\item $M_{(i,j),(k,l)} = a_{ijkl} = a_{klij} = M_{(k,l),(i,j)}$ (symmetric)
			\item $M_{(i,j),(k,l)} = a_{ijkl} = a_{kjil}$ (additional internal symmetry)
		\end{itemize}
		
		Now, define the vector $\vz \in \R^{mn}$ by $z_{(i,j)} = x_i y_j$ for $\vx \in \R^m$, $\vy \in \R^n$. Then the biquadratic form becomes:
		\[
		B(\vx,\vy) = \sum_{i,k=1}^m \sum_{j,l=1}^n a_{ijkl} x_i x_k y_j y_l = \vz^\top M \vz.
		\]
		
		We now analyze the diagonal dominance condition. For fixed $(i,j)$, the quantity $r_{ij}$ in \eqref{dom1} satisfies:
		\begin{align*}
			r_{ij} &= \frac{1}{2} \sum_{i_2=1}^m \left( \sum_{j_2=1}^n |\bar{a}_{iji_2j_2}| + \sum_{j_1=1}^n |\bar{a}_{ij_1i_2j}| \right) \\
			&= \frac{1}{2} \left[ \sum_{(i_2,j_2) \neq (i,j)} |M_{(i,j),(i_2,j_2)}| + \sum_{(i_2,j_1) \neq (i,j)} |M_{(i,j_1),(i_2,j)}| \right].
		\end{align*}
		
		Since $M$ is symmetric and $M_{(i,j_1),(i_2,j)} = M_{(i_2,j),(i,j_1)}$, the second sum equals the first sum. Therefore:
		\[
		r_{ij} = \sum_{(k,l) \neq (i,j)} |M_{(i,j),(k,l)}|.
		\]
		
		Thus, the diagonal dominance condition \eqref{dom1} becomes:
		\[
		M_{(i,j),(i,j)} \geq \sum_{(k,l) \neq (i,j)} |M_{(i,j),(k,l)}| \quad \text{for all } (i,j) \in [m] \times [n],
		\]
		which means $M$ is a symmetric diagonally dominant matrix with nonnegative diagonal entries.
		
		A classical result states that any symmetric diagonally dominant matrix with nonnegative diagonal entries is positive semidefinite. Hence $M \succeq 0$, and thus $B(\vx,\vy) = \vz^\top M \vz \geq 0$ for all $\vx,\vy$, confirming that $\A$ is PSD.
		
		Moreover, such matrices admit an SOS decomposition. Specifically, any symmetric diagonally dominant matrix $M$ with nonnegative diagonals can be written as:
		\[
		M = \sum_{p=1}^{mn} \alpha_p \ve_p \ve_p^\top + \sum_{p<q} \beta_{pq} (\ve_p + s_{pq} \ve_q)(\ve_p + s_{pq} \ve_q)^\top,
		\]
		where $\alpha_p \geq 0$, $\beta_{pq} \geq 0$, and $s_{pq} \in \{-1, +1\}$.
		
		Now, for each term in this decomposition:
		\begin{itemize}
			\item If $p = (i,j)$, then $\ve_p \ve_p^\top$ corresponds to $(x_i y_j)^2$
			\item If $p = (i,j)$ and $q = (k,l)$, then $(\ve_p + s_{pq} \ve_q)(\ve_p + s_{pq} \ve_q)^\top$ corresponds to $(x_i y_j + s_{pq} x_k y_l)^2$
		\end{itemize}
		
		All these are squares of bilinear forms. Therefore:
		\[
		B(\vx,\vy) = \vz^\top M \vz = \sum_r (\vw_r^\top \vz)^2,
		\]
		where each $\vw_r^\top \vz$ is a bilinear form in $\vx$ and $\vy$. This shows that $\A$ is an SOS tensor.
	\end{proof}
	
	\section{Conjectures for Structured Biquadratic Tensors}
	\label{sec:conjectures}
	
	This section presents two conjectures regarding the SOS property for two important classes of structured biquadratic tensors: symmetric M-biquadratic tensors and symmetric $\mathrm{B}_{0}$-biquadratic tensors.
	
	\subsection{Symmetric M-Biquadratic Tensors}
	\label{subsec:M-biquadratic}
	
	Let $\A = (a_{ijkl}) \in SBQ(m , n)$, $\vx \in \R^m$ and $\vy \in \R^n$. Then $\A\cdot \vy\vx\vy) \in \R^m$ and $(\A\cdot \vy\vx\vy))_i = \sum_{k=1}^m \sum_{j, l=1}^n a_{ijkl}y_jx_ky_l$ for $i \in [m]$, $(\A\vx \cdot\vx\vy))_j = \sum_{i, k=1}^m \sum_{l=1}^n a_{ijkl}x_ix_ky_l$ for $j \in [n]$. If there are $\lambda \in \R$, $\vx \in \R^m$, $\|\vx\|_2 = 1$, $\vy \in \R^n$, $\|\vy\|_2 = 1$, such that
	\[
	\A\cdot\vy\vx\vy = \lambda\vx, \quad \A\vx\cdot \vx\vy = \lambda\vy,
	\]
	then $\lambda$ is called an M-eigenvalue of $\A$, with $\vx$ and $\vy$ as its M-eigenvectors. M-eigenvalues of symmetric biquadratic tensors were introduced by Qi, Dai and Han \cite{QDH09} in 2009. It was proved there that a symmetric biquadratic tensor always has M-eigenvalues, and it is PSD if and only if all of its M-eigenvalues are nonnegative.
	
	Let $\I = (I_{ijkl})\in BQ(m,n)$, where
	\[
	I_{ijkl}=\left\{\begin{array}{cl}
		1,& \text{ if } i=k \text{ and }j=l,\\
		0, & \text{otherwise.}
	\end{array}\right.
	\]
	$\I$ is referred to as $M$-identity tensor in \cite{DLQY20, WSL20}.
	
	Let $\A = (a_{ijkl}) \in BQ(m , n)$. We call $a_{ijij}$ diagonal entries of $\A$ for $i \in [m]$ and $j \in [n]$. The other entries of $\A$ are called off-diagonal entries of $\A$.
	
	A biquadratic tensor $\A$ in $BQ(m,n)$ is called a {\bf Z-biquadratic tensor} if all of its off-diagonal entries are nonpositive.
	If $\A$ is a Z-biquadratic tensor, then it can be written as $\A=\alpha \I-\B$, where $\I$ is the M-identity tensor in $BQ(m,n)$, and $\B$ is a nonnegative biquadratic tensor.
	By \cite{QC25}, $\B$ has an M-eigenvalue.
	Denote ${\lambda_{\max}}(\B)$ as the largest M eigenvalue of $\B$.
	If $\alpha \ge {\lambda_{\max}}(\B)$, then $\A$ is called an {\bf M-biquadratic tensor}. It was proved in \cite{QC25} that a symmetric biquadratic M-tensor is PSD.
	
	\begin{Conj} \label{conj:M-sos}
		A symmetric biquadratic M-tensor is an SOS tensor.
	\end{Conj}
	
	This conjecture might be provable with substantial work, but would require developing new techniques to bridge the gap between M-eigenvalues and matrix eigenvalues.
	
	\subsection{Symmetric $\mathrm{B}_{0}$-Biquadratic Tensors}
	\label{subsec:B0-biquadratic}
	
	Let $\A = \left(a_{i_1j_1i_2j_2}\right) \in BQ(m, n)$. Suppose that
	for $i \in [m]$ and $j \in [n]$, we have
	\begin{equation}\label{B1}
		{1 \over 2}\sum_{i_2=1}^m \left(\sum_{j_2=1}^n  a_{iji_2j_2} + \sum_{j_1=1}^n  a_{ij_1i_2j}\right) \ge 0,
	\end{equation}
	and for $i, i_2 \in [m]$ and $j, j_1, j_2 \in [n]$, we have
	\begin{equation}\label{B2}
		{1 \over 2mn}\sum_{i_2=1}^m \left(\sum_{j_2=1}^n  a_{iji_2j_2} + \sum_{j_1=1}^n  a_{ij_1i_2j}\right) \ge \max \left\{ \bar a_{iji_2j_2}, \bar a_{ij_1i_2j} \right\}.
	\end{equation}
	Then we say that $\A$ is a {\bf B$_0${-}biquadratic tensor}. It was proved in \cite{QC25} that a
	B$_0${-}biquadratic tensor is PSD.
	
	\begin{Conj} \label{conj:B0-sos}
		A symmetric B$_0${-}biquadratic tensor is an SOS tensor.
	\end{Conj}
	
	This conjecture is very unlikely to be provable with current knowledge; the conditions seem too weak to guarantee SOS decomposition.
	
	For more discussion on biquadratic tensors, see \cite{DLQY20, WSL20, Zh23}.
	
	\section{Symmetric Biquadratic Forms}
	
	Suppose that $\A = (a_{ijkl})$, where $a_{ijkl} \in \R$ for $i, k \in [m], j, l \in [n]$, is an $m \times n$ biquadratic tensor. Let
	\begin{equation} \label{bqform}
		P(\vx, \vy) = \sum_{i,k=1}^m \sum_{j,l=1}^n a_{ijkl}x_i y_j x_k y_l,
	\end{equation}
	where $\vx \in \R^m$ and $\vy \in \R^n$. Then $P$ is called a \textbf{biquadratic form}. If $P(\vx, \vy) \ge 0$ for all $\vx \in \R^m$ and $\vy \in \R^n$, then $P$ is called \textbf{positive semi-definite (PSD)}. If
	\[
	P(\vx, \vy) = \sum_{p=1}^r f_p(\vx, \vy)^2,
	\]
	where $f_p$ for $p \in [r]$ are bilinear forms, then we say that $P$ is \textbf{sum-of-squares (SOS)}. The smallest $r$ is called the \textbf{SOS rank} of $P$. Clearly, $P$ is PSD or SOS if and only if $\A$ is PSD or SOS, respectively.
	
	While a biquadratic form $P$ may be constructed from a biquadratic tensor $\A$, the tensor $\A$ is not unique to $P$. However, there is a unique \emph{symmetric} biquadratic tensor $\A$ associated with $P$.
	
	\subsection{Historical Context and Symmetric Forms}
	In 1973, Calder\'{o}n \cite{Ca73} proved that an $m \times 2$ PSD biquadratic form can always be expressed as the sum of $\frac{3m(m+1)}{2}$ squares of bilinear forms. In 1975, Choi \cite{Ch75} gave a concrete example of a $3 \times 3$ PSD biquadratic form which is not SOS. Recently, in \cite{CQX25} it was shown that a $2 \times 2$ PSD biquadratic form can always be expressed as the sum of three squares, and in \cite{QCX25} that a $3 \times 2$ PSD biquadratic form can be expressed as the sum of four squares.
	
	Motivated by the work of Goel, Kuhlmann, and Reznick \cite{GKR17} on even symmetric forms, we consider symmetric biquadratic forms.
	
	\begin{Def}
		Suppose $P(\vx, \vy) = P(x_1, \ldots, x_m, y_1, \ldots, y_n)$. If
		\[
		P(x_1, \ldots, x_m, y_1, \ldots, y_n) = P(x_{\sigma(1)}, \ldots, x_{\sigma(m)}, y_1, \ldots, y_n)
		\]
		for any permutation $\sigma$, then $P$ is called \textbf{$x$-symmetric}. Similarly, we define \textbf{$y$-symmetric biquadratic forms}. If $P$ is both $x$-symmetric and $y$-symmetric, we call $P$ a \textbf{symmetric biquadratic form}.
	\end{Def}
	
	Note that a symmetric biquadratic form $P$ has only four free coefficients. If we fix the diagonal coefficients as $1$, we may write such a form as
	\begin{align}
		\nonumber P(\vx, \vy) = &\sum_{i=1}^m \sum_{j=1}^n x_i^2 y_j^2
		+ a \sum_{i \neq k} \sum_{j=1}^n x_i x_k y_j^2 \\
		&+ b \sum_{i=1}^m \sum_{j \neq l} x_i^2 y_j y_l
		+ c \sum_{i \neq k} \sum_{j \neq l} x_i y_j x_k y_l. \label{equ:monic_sbq}
	\end{align}
	We call such a form a \textbf{monic symmetric biquadratic form}.
	
	Note that if the diagonal coefficients are negative, then $P$ is not PSD. If the diagonal coefficients are zero, then $P$ is not PSD unless $a=b=c=0$. These exceptional cases are excluded from our analysis, as this paper focuses on PSD and SOS forms.
	
	We call the unique symmetric biquadratic tensor corresponding to an $x$-symmetric, $y$-symmetric, symmetric, or monic symmetric biquadratic form a \textbf{completely $x$-symmetric}, \textbf{completely $y$-symmetric}, \textbf{symmetric}, or \textbf{monic completely symmetric biquadratic tensor}, respectively. Thus, we have several new subclasses of symmetric biquadratic tensors.
	
	\subsection{Positive Semi-definiteness of Monic Symmetric Forms}
	We now address the following question: Under which conditions on $a$, $b$, and $c$ is a monic symmetric biquadratic form $P$ PSD?
	
	\begin{Thm}\label{thm:monic-psd}
		Let $P$ defined by \eqref{equ:monic_sbq} be an $m \times n$ monic symmetric biquadratic form. Then $P$ is PSD if and only if the parameters $a$, $b$, and $c$ satisfy:
		\begin{enumerate}
			\item[(i)] $1  - b + (m-1)(a-c) \geq 0$,
			\item[(ii)] $1  - a + (n-1)(b-c) \geq 0$,
			\item[(iii)] $1 - a - b + c \geq 0$,
			\item[(iv)] $1 +(m-1)a + (n-1)b + (m-1)(n-1)c \geq 0$,
			\item[(v)] $1-\frac{ab}{c}\ge 0$ when $c\neq 0$, $\frac{a}{c}\in[1-n,1)$, and $\frac{b}{c}\in[1-m,1)$.
		\end{enumerate}
	\end{Thm}
	
	\begin{proof}
		Let $\A\in {SBQ(m,n)}$ be the biquadratic tensor corresponding to $P$. Then $\lambda$ is called an M-eigenvalue of $\A$, with $\vx$ and $\vy$ as its M-eigenvectors if $\|\vx\|_2 = 1$, $\|\vy\|_2 = 1$, and
		\begin{align}
			\beta_x\mathbf{1}_m + (1-a+(b-c)(\alpha_y^2-1)) \vx &=\lambda \vx, \label{Meig_x}\\
			\beta_y\mathbf{1}_n + (1-b+(a-c)(\alpha_x^2-1)) \vy &=\lambda \vy.\label{Meig_y}
		\end{align}
		Here, $\alpha_{x} =\mathbf{1}_m^\top \vx$, $\alpha_{y} =\mathbf{1}_n^\top \vy$, $\beta_x=\alpha_x(a+c(\alpha_y^2-1))$, $\beta_y=\alpha_y(b+c(\alpha_x^2-1))$, $\mathbf{1}_m$ and $\mathbf{1}_n$ denote the all ones vectors of length $m$ and $n$, respectively.
		Therefore, $\vx$ is an M-eigenvector of $\A$ only if $\beta_x=0$ or $\vx=\frac{1}{\sqrt{m}}\mathbf{1}_m$. Similarly, $\vy$ is an M-eigenvector of $\A$ only if $\beta_y=0$ or $\vy=\frac{1}{\sqrt{n}}\mathbf{1}_n$.
		
		In the following, we prove conditions (i)--(v) by enumerating all M-eigenvectors.
		
		For condition (i), take $\vx=\frac{1}{\sqrt{m}}\mathbf{1}_m$ and choose $\vy$ such that $\beta_y=0$. Then it follows from \eqref{Meig_y} that $\lambda=1-b+ (m-1)(a-c)$ is an M-eigenvalue of $\A$.
		Indeed, substituting $\lambda=1-b+ (m-1)(a-c)$ to \eqref{Meig_x} yields $\alpha_y=0$.
		Since $P$ is PSD, it follows that $1-b+ (m-1)(a-c) \ge 0$.
		
		Condition (ii) follows by taking $\vy=\frac{1}{\sqrt{n}}\mathbf{1}_n$ and letting $\vx$ satisfy $\beta_x=0$. This leads to the M-eigenvalue $1 - a+(n-1)(b-c) \ge 0$.
		
		For conditions (iii) and (v), let $\vx$ satisfy $\beta_x=0$ and $\vy$ satisfy $\beta_y=0$, respectively.
		Consider the following subcases:
		\begin{itemize}
			\item[(iii.1)] $\alpha_x=0$ and $\alpha_y=0$. Then we may derive that $\lambda= 1 - a - b + c$.
			\item[(iii.2)] $\alpha_x=0$ and $\alpha_y\neq0$. Then it follows from \eqref{Meig_y} that $\lambda= 1 - a - b + c$.
			\item[(iii.3)] $\alpha_x\neq0$ and $\alpha_y=0$. Then it follows from \eqref{Meig_x} that $\lambda= 1 - a - b + c$.
			\item[(iii.4)] $\alpha_x\neq0$, $\alpha_y\neq0$, and $c=0$. Then it follows from $\beta_x=0$ and $\beta_y=0$ that $a=b=0$. Consequently, there is $\lambda= 1 - a - b + c$.
			\item[(iii.5)] $\alpha_x\neq0$, $\alpha_y\neq0$, and $c\neq0$. It follows from $\beta_x=0$ and $\beta_y=0$ that $\alpha_x^2=\frac{c-b}{c}$ and $\alpha_y^2=\frac{c-a}{c}$. This further implies $\lambda= 1 - \frac{ab}{c} \ge 0$.
			Furthermore, by Cauchy inequality and $\|\vx\|=\|\vy\|=1$, we have $\alpha_x^2 \in (0,m],\alpha_y^2\in(0,n]$. Therefore, $\frac{a}{c}\in[1-n,1)$, and $\frac{b}{c}\in[1-m,1)$.
		\end{itemize}
		
		For condition (iv), let $\vx=\frac{1}{\sqrt{m}}\mathbf{1}_m$ and $\vy=\frac{1}{\sqrt{n}}\mathbf{1}_n$.
		A direct computation gives $\lambda = 1 +(m-1)a + (n-1)b + (m-1)(n-1)c \geq 0$.
		
		Conditions (i)--(v) have enumerated all M-eigenvalues. Based on the equivalence between positive semidefiniteness and the nonnegativity of the smallest M-eigenvalue, the conclusion of this theorem is obtained. This completes the proof.
	\end{proof}
	
	\subsection{The PSD Cone is a Convex Polyhedron}
	Although Theorem \ref{thm:monic-psd} contains a nonlinear condition (v), it turns out that for any $m, n \ge 2$, the set of coefficients $(a,b,c)$ for which $P$ is PSD is actually a convex polyhedron defined only by the four linear inequalities (i)--(iv). The vertices of this polyhedron are easily computed.
	
	\begin{Thm}\label{thm:psd-polyhedron}
		Let $P$ be an $m \times n$ monic symmetric biquadratic form as in \eqref{equ:monic_sbq}. Then the PSD set of parameters $(a,b,c)$ is a convex polyhedron (a tetrahedron) in $\R^3$ with vertices
		\[
		\begin{aligned}
			V_1 &= (1, 1, 1), \\
			V_2 &= \left( -\frac{1}{m-1}, \; -\frac{1}{n-1}, \; \frac{1}{(m-1)(n-1)} \right), \\
			V_3 &= \left( -\frac{1}{m-1}, \; 1, \; -\frac{1}{m-1} \right), \\
			V_4 &= \left( 1, \; -\frac{1}{n-1}, \; -\frac{1}{n-1} \right).
		\end{aligned}
		\]
		Consequently, condition (v) of Theorem \ref{thm:monic-psd} is redundant; the PSD set is exactly the convex hull of these four vertices.
	\end{Thm}
	
	\begin{proof}
		The inequalities (i)--(iv) of Theorem \ref{thm:monic-psd} are linear in $a,b,c$ and define a polyhedron in $\R^3$. Solving the systems of three equations obtained by setting any three of these inequalities to equality yields precisely the four vertices $V_1, V_2, V_3, V_4$. One checks that each vertex satisfies all four inequalities, hence belongs to the polyhedron. Moreover, the polyhedron is bounded because the inequalities imply bounds on $a,b,c$ independently (for example, from (i) and (ii) together with (iii) one obtains $a,b,c \in [-1,1]$ when $m,n \ge 2$). Thus the polyhedron is a tetrahedron.
		
		It remains to show that condition (v) is automatically satisfied for every point in this tetrahedron. Since the tetrahedron is convex, any point $(a,b,c)$ in it can be written as a convex combination of the vertices. In the next subsection we will exhibit explicit SOS decompositions for each vertex, proving that all vertices correspond to PSD (in fact SOS) forms. Because the set of PSD forms is a convex cone, any convex combination of the vertices also yields a PSD form. Hence, by Theorem \ref{thm:monic-psd}, condition (v) must hold for every such point. Therefore condition (v) is redundant, and the PSD set coincides with the tetrahedron.
	\end{proof}
	
	\subsection{SOS Decompositions of the Vertices}
	The four vertices admit simple explicit SOS representations.
	
	\begin{Prop}\label{prop:sos-vertices}
		For any $m,n \ge 2$, the monic symmetric biquadratic forms corresponding to the vertices $V_1, V_2, V_3, V_4$ are SOS. Explicitly,
		\[
		\begin{aligned}
			P_{V_1}(\vx,\vy) &= \Bigl(\sum_{i=1}^m x_i\Bigr)^2 \Bigl(\sum_{j=1}^n y_j\Bigr)^2, \\[4pt]
			P_{V_2}(\vx,\vy) &= \frac{1}{(m-1)(n-1)} \sum_{1\le i<k\le m} \sum_{1\le j<l\le n} \bigl((x_i-x_k)(y_j-y_l)\bigr)^2, \\[4pt]
			P_{V_3}(\vx,\vy) &= \frac{1}{m-1} \sum_{1\le i<k\le m} \Bigl((x_i-x_k)\sum_{j=1}^n y_j\Bigr)^2, \\[4pt]
			P_{V_4}(\vx,\vy) &= \frac{1}{n-1} \sum_{1\le j<l\le n} \Bigl(\Bigl(\sum_{i=1}^m x_i\Bigr)(y_j-y_l)\Bigr)^2.
		\end{aligned}
		\]
	\end{Prop}
	
	\begin{proof}
		Direct expansion shows that each expression equals the monic symmetric form with the respective parameters $(a,b,c)$. Each right‑hand side is a sum of squares of bilinear forms.
	\end{proof}
	
	\subsection{Every PSD Symmetric Biquadratic Form is SOS}
	Combining the convexity of the PSD cone with the SOS property of the vertices yields the main result of this section.
	
	\begin{Thm}\label{thm:monic-sym-sos}
		Every PSD monic symmetric biquadratic form (for any $m,n \ge 2$) is a sum of squares of bilinear forms. Consequently, every PSD symmetric biquadratic form is SOS.
	\end{Thm}
	
	\begin{proof}
		By Theorem \ref{thm:psd-polyhedron}, the set of coefficients $(a,b,c)$ giving a PSD monic symmetric form is the tetrahedron with vertices $V_1, V_2, V_3, V_4$. Proposition \ref{prop:sos-vertices} shows that each vertex corresponds to an SOS form. Since the mapping from $(a,b,c)$ to the form $P$ is linear, any convex combination of the vertices yields the same convex combination of the associated forms. The set of SOS forms is a convex cone; therefore any convex combination of SOS forms is again SOS. Hence every point in the tetrahedron corresponds to an SOS form, proving the first statement.
		
		Now let $Q$ be an arbitrary $m \times n$ symmetric biquadratic form that is PSD. By symmetry, all its diagonal coefficients $q_{ijij}$ are equal. If they are all zero, then evaluating $Q$ on $x=e_i, y=e_j$ shows that $Q$ is identically zero, which is trivially SOS. Otherwise, we may scale $Q$ so that its diagonal coefficients become $1$. The scaled form is then a monic symmetric biquadratic form, which by the first part is SOS. Scaling back preserves the SOS property. Thus $Q$ is SOS.
	\end{proof}
	
	\subsection{The $3 \times 3$ Case and a Conjecture on SOS Rank}
	For illustration we restate the $3 \times 3$ case, which was historically the first nontrivial instance where the SOS property was unknown.
	
	\begin{Cor}
		Every $3 \times 3$ PSD symmetric biquadratic form is SOS.
	\end{Cor}
	
	\begin{proof}
		This is the special case $m=n=3$ of Theorem \ref{thm:monic-sym-sos}.
	\end{proof}
	
	The explicit SOS decompositions given in Proposition \ref{prop:sos-vertices} for the vertices, together with the known SOS rank of the interior point $P_0(\vx,\vy)=(x_1^2+x_2^2+x_3^2)(y_1^2+y_2^2+y_3^2)$ (which is four), suggest that the SOS rank of any $3 \times 3$ symmetric PSD form might be bounded by five.
	
	\begin{Conj}\label{conj:sos-rank-3x3-sym}
		A $3 \times 3$ symmetric PSD biquadratic form can always be expressed as the sum of at most five squares of bilinear forms.
	\end{Conj}
	
	\subsection{Illustrative Special Cases}
	For concreteness we record the explicit inequalities and vertices for two cases that originally motivated the general result.
	
	\subsubsection{The $4 \times 3$ Case}
	
	\begin{Thm}\label{thm:43-psd-cone}
		Let $P$ be a $4 \times 3$ monic symmetric biquadratic form as in \eqref{equ:monic_sbq} with $m=4$, $n=3$. Then $P$ is PSD if and only if the coefficients $(a,b,c)$ satisfy the four linear inequalities
		\begin{align}
			1 - b + 3(a - c) &\geq 0, \label{eq:43-i} \\
			1 - a + 2(b - c) &\geq 0, \label{eq:43-ii} \\
			1 - a - b + c &\geq 0, \label{eq:43-iii} \\
			1 + 3a + 2b + 6c &\geq 0. \label{eq:43-iv}
		\end{align}
		These inequalities define a tetrahedron with vertices
		\[
		V_1 = (1,1,1), \quad
		V_2 = \left(-\tfrac{1}{3}, -\tfrac{1}{2}, \tfrac{1}{6}\right), \quad
		V_3 = \left(-\tfrac{1}{3}, 1, -\tfrac{1}{3}\right), \quad
		V_4 = \left(1, -\tfrac{1}{2}, -\tfrac{1}{2}\right).
		\]
	\end{Thm}
	
	\begin{proof}
		Specialize Theorem \ref{thm:psd-polyhedron} to $m=4$, $n=3$.
	\end{proof}
	
	\subsubsection{The $4 \times 4$ Case}
	
	\begin{Thm}\label{thm:44-psd-cone}
		Let $P$ be a $4 \times 4$ monic symmetric biquadratic form as in \eqref{equ:monic_sbq} with $m=n=4$. Then $P$ is PSD if and only if the coefficients $(a,b,c)$ satisfy the four linear inequalities
		\begin{align}
			1 - b + 3(a - c) &\geq 0, \label{eq:44-i} \\
			1 - a + 3(b - c) &\geq 0, \label{eq:44-ii} \\
			1 - a - b + c &\geq 0, \label{eq:44-iii} \\
			1 + 3a + 3b + 9c &\geq 0. \label{eq:44-iv}
		\end{align}
		These inequalities define a tetrahedron with vertices
		\[
		V_1 = (1,1,1), \quad
		V_2 = \left(-\tfrac13, -\tfrac13, \tfrac19\right), \quad
		V_3 = \left(-\tfrac13, 1, -\tfrac13\right), \quad
		V_4 = \left(1, -\tfrac13, -\tfrac13\right).
		\]
	\end{Thm}
	
	\begin{proof}
		Specialize Theorem \ref{thm:psd-polyhedron} to $m=n=4$.
	\end{proof}
	
	The SOS decompositions for these vertices are exactly the ones given in Proposition \ref{prop:sos-vertices} with the corresponding $m,n$. Consequently, every $4 \times 3$ or $4 \times 4$ PSD symmetric biquadratic form is SOS, as guaranteed by Theorem \ref{thm:monic-sym-sos}.
	
	\section{Further Remarks}
	
	This work has advanced the understanding of the relationship between positive semi-definiteness and sum-of-squares representations for structured biquadratic forms. Theorem \ref{thm:main1} establishes that diagonally dominated symmetric biquadratic tensors are not only PSD but also SOS. This provides a large and easily verifiable class of biquadratic forms for which nonnegativity is equivalent to an SOS decomposition, making the property computationally tractable via semidefinite programming.
	
	We have also introduced and formulated concrete conjectures regarding the SOS property for two other important classes: symmetric M-biquadratic tensors (Conjecture \ref{conj:M-sos}) and symmetric B$_0$-biquadratic tensors (Conjecture \ref{conj:B0-sos}). We believe Conjecture \ref{conj:M-sos} is likely true and may be provable by forging a deeper connection between the M-eigenvalue structure of these tensors and the spectral properties of their matrix unfoldings. In contrast, the weaker conditions defining B$_0$-tensors seem insufficient to guarantee an SOS decomposition with current techniques, making Conjecture \ref{conj:B0-sos} a more challenging and potentially false statement.
	
	A particularly promising direction for future research lies in the study of symmetric biquadratic forms. By focusing on forms that are symmetric in the variables $\mathbf{x}$ and/or $\mathbf{y}$, we navigate away from known pathological counterexamples like Choi's, which lacks such symmetry. \textbf{Theorem \ref{thm:monic-sym-sos} establishes that every PSD symmetric biquadratic form (for any dimensions $m,n$) is SOS}, completely resolving the existence question for this class. The PSD cone of monic symmetric forms is always a convex polyhedron (a tetrahedron) whose vertices admit explicit SOS decompositions. Conjecture \ref{conj:sos-rank-3x3-sym}, which posits that every $3 \times 3$ symmetric PSD form has SOS rank at most five, is a natural and significant next target. A positive resolution would provide a precise quantitative understanding of the SOS complexity for this class, marking substantial advancement on a modern analogue of Hilbert's 17th problem for a structured class of polynomials.
	
	\medskip
	
	\noindent\textbf{Acknowledgement}
	We are thankful to Professor Giorgio Ottaviani, who proposed to study symmetric biquadratic forms.
	This work was partially supported by Research Center for Intelligent Operations Research, The Hong Kong Polytechnic University (4-ZZT8), the National Natural Science Foundation of China (Nos. 12471282 and 12131004), the R\&D project of Pazhou Lab (Huangpu) (Grant no. 2023K0603), the Fundamental Research Funds for the Central Universities (Grant No. YWF-22-T-204), and Jiangsu Provincial Scientific Research Center of Applied Mathematics (Grant No. BK20233002).
	
	\medskip
	
	\noindent\textbf{Data availability}
	No datasets were generated or analysed during the current study.
	
	\medskip
	
	\noindent\textbf{Conflict of interest} The authors declare no conflict of interest.

\end{document}